\documentclass[10pt]{amsart}
\usepackage{amssymb, amscd, amsmath, amsthm, epsfig, latexsym, enumerate}

\renewcommand{\geq}{\geqslant}
\renewcommand{\leq}{\leqslant}

\newcommand{\Z}{\mathbb Z}

\DeclareMathOperator{\cd}{cd}

\newtheorem{theorem}{Theorem}
\newtheorem{lemma}[theorem]{Lemma}
\newtheorem{cor}[theorem]{Corollary}

\newtheorem*{thm}{Theorem}

\newtheorem*{cor*}{Corollary}

\begin{document}
\title{Aspherical $PD_3$-pairs}

\author{Jonathan A. Hillman }
\address{School of Mathematics and Statistics\\
     University of Sydney, NSW 2006\\
      Australia }

\email{jonathanhillman47@gmail.com}

\begin{abstract}
We extend two results known for aspherical 3-manifolds to $PD_3$-pairs $(P,\partial{P})$ 
with aspherical ambient space $P$.
Every such $PD_3$-pair may be assembled by attaching 1-handles to $PD_3$-pairs 
with aspherical ambient space and $\pi_1$-injective boundary,
and copies of $(D^3,S^2)$.
(Thus the study of such pairs reduces to the study of $PD_3$-pairs of groups.)
If $\pi$ is a group of type $FP$ whose indecomposable factors $G_i$ each have $\chi(G_i)=0$
then there are only finitely many such $PD_3$-pairs with $\pi_1(P)\cong\pi$.
\end{abstract}


\maketitle
A $PD_3$-pair $(P,\partial{P})$ with peripheral system $(\pi,\{\kappa_j\})$  
and orientation character $w:\pi\to\Z^\times$
is determined up to homotopy equivalence by its fundamental triple $((\pi,\{\kappa_j\}), w,\mu)$,  
where $\mu\in{H_3(\pi,\{\kappa_j\};\Z^w})$ is the image of a fundamental class,
by the Classification Theorem of Hendriks \cite[Theorem 3.2]{PDD3}.
The pair $(P,\partial{P})$ is equivalent to a $PD_3$-pair of groups if and only if 
$\pi$ has one end and the peripheral system is injective.
 
In \cite[Chapter 13]{PDD3} we showed that if $\pi$ has a ``large enough" free factor 
and aspherical boundary components then $(P,\partial{P})$ may be assembled 
from $PD_3$-complexes (with empty boundary) and pairs equivalent to $PD_3$-pair of groups, 
by forming connected sums and adding 1-handles.
The basic idea was to show that the peripheral systems could be decomposed into 
sums of simpler systems which inherited the boundary compatibility and Turaev conditions
that ensured that they could be realized by $PD_3$-pairs, at least if the peripheral system is injective,
and then to use the Classification Theorem to identify the pair assembled from 
such pieces with the original pair.

In the first of our two main results, Theorem \ref{asph wellb},
we show that if the ambient space $P$ is aspherical then the pair $(P,\partial{P})$ 
may be assembled from pairs equivalent to $PD_3$-pairs of groups 
by adding 1-handles (which includes forming boundary connected sums).
We  shall use some of the steps from \cite[Chapter 13]{PDD3}, 
but avoid the extra hypothesis on free factors, 
by using \cite[Theorem 8]{DH3} to see that the resulting fragments are $PD_3$-pairs.
(Our present argument does seem to need asphericity of $P$.)

The first section summarizes the relevant terminology from \cite{PDD3}.
The asphericity hypotheses imply that the homotopy type of the pair is determined
 by its peripheral group system.
This is used in \S2 to identify candidates for the fragments of the pair corresponding to 
the decomposition of the group, and then to see that the pair
reassembled from the fragments is homotopy equivalent to the original pair. 
Theorem \ref{asph wellb} represents a sharpening of some of the results of Chapter 13 of \cite{PDD3},
and provides an analogue of the Loop Theorem for aspherical $PD_3$-pairs, 
strengthening the Algebraic Loop Theorem of Crisp \cite[Theorem 3.10]{PDD3}.
(The Algebraic Loop Theorem is essential for our argument.)

In \S3 we give our second main result, Theorem \ref{finitely many},
in which we show that under more stringent conditions
(corresponding to an aspherical  3-manifold with boundary components tori or Klein bottles),
the pair $(P,\partial{P})$ is determined up to a finite ambiguity by the group $\pi$ alone.
This result extends Theorem 3.14 of \cite{PDD3}, 
which assumed also that $\pi$ had one end and the group system was atoroidal.
(The corresponding result for irreducible 3-manifolds does not need the restriction on the
boundary components \cite{Sw80}.)

\section{preliminaries}

We shall say that a $PD_3$-pair $(P,\partial{P})$ is {\it aspherical\/} if $P$ is aspherical.

\subsection*{Adding 1-handles}

If $(U,V)$ is a CW pair with $V$ a surface then we may {\it add a 1-handle} $h^1=D^1\times{D^2}$
by identifying $\partial_-h^1=S^0\times{D^2}$ with a pair of disjoint $2$-discs in $V$,
to obtain a new pair $(U',V')=(U,V)\cup{h^1}$, with $U'\simeq{U}\vee{S^1}$.
(There are choices of components of $U$ and $V$ and local orientations 
to be made,  but we will not complicate the notation!)
If the components of $\partial_-h^1$ lie in the same component of $V$ then $V'=V\#T$ or $V\#Kb$,
and we say that $(U',V')=(U,V)\natural(E,\partial{E})$ is a {\it boundary connected sum}, 
with $E$ the total space of a $D^2$-bundle over $S^1$.
If they lie in different components of $U$ then $V'$ is obtained from $V$ by forming the 
internal connected sum of two of its components, and we again say that $(U',V')$ 
is the result of a boundary-connected sum.
There is a third possibility, 
in which the components of $\partial_-h^1$ lie in the same
component of $U$ but in different components of $V$.

We shall use the following variant of \cite[Theorem 2.1]{Wa67}.

\begin{thm}
\cite[Theorem 8]{DH3}
If $(U, V)$ is a $CW$-pair such that $V$ is a $PD_{n-1}$-complex
and  the pair $(U',V')=(U,V)\cup{h^1}$ obtained by attaching a $1$-handle $h^1=D^1\times{D^{n-1}}$ 
(with $\partial_-h^1=S^0\times{D^{n-1}}\subset{V}$) is a $PD_n$-pair then $(U,V)$ is a $PD_n$-pair.
\qed
\end{thm}

In our application below (Theorem \ref{asph wellb}), 
we start with an aspherical  $PD_3$-pair, corresponding to $(U',V')$,
and use asphericity to identify a $CW$-pair corresponding to $(U,V)$.
Theorem 8 of \cite{DH3} then ensures that the new pair is again a $PD_3$-pair.

\subsection*{Geometric systems and their decompositions}

Let $Y$ be an aspherical closed surface and $G$ be a group.
A homomorphism $\kappa:S=\pi_1(Y)\to{G}$ is {\it geometric\/} 
if there is a finite family $\Phi$ of disjoint, 
2-sided simple closed curves on $Y$ such that $\mathrm{Ker}(\kappa)$ 
is normally generated by the image of $\Phi$ in $S$.
A {\it geometric group system\/} is a pair $(G,\mathcal{K})$, 
where $G$ is a finitely presentable group and $\mathcal{K}$ is 
a finite set of geometric homomorphisms from $PD_2$-groups to $G$.
It is {\it trivial\/} if $G=1$.
(The set $\mathcal{K}$ is then empty.)
Two geometric group systems $(G_1,\mathcal{K}_1)$ and 
$(G_2,\mathcal{K}_2)$  are {\it isomorphic\/} if there is an isomorphism 
$\theta:G_1\to{G_2}$ and a bijection $b:\mathcal{K}_1\to\mathcal{K}_2$
such that $\theta\circ\kappa$ and $b(\kappa)$ have the same domains
and are conjugate as homomorphisms into $G_2$, 
for all $\kappa\in\mathcal{K}_1$.
A homomorphism $w:G\to\mathbb{Z}^\times$ is an {\it orientation character\/} for $(G,\mathcal{K})$ 
if $w\circ\kappa=w_1(Y)$ for each $\kappa\in\mathcal{K}$.
The peripheral system of a $PD_3$-pair with aspherical boundary is a geometric group system, 
by the Algebraic Loop Theorem \cite[Theorem 3.10]{PDD3}.

There are two versions of connected sum for $PD_3$-pairs, 
corresponding to the internal and boundary connected sums of manifolds.
The peripheral system of an internal connected sum is the free product of
the peripheral systems of the factors, 
while that of the boundary connected is slightly more complicated.
When we use the term ``connected sum" without further qualification, 
we shall allow both possibilities.
In each case, the fundamental group of the ambient space is a free product.

The sum of two geometric group systems $(G_1,\mathcal{K}_1)$ and 
$(G_2,\mathcal{K}_2)$ is the geometric group system
\[
(G_1,\mathcal{K}_1)\sharp(G_2,\mathcal{K}_2)=
(G_1*G_2,\mathcal{K}_1\sqcup\mathcal{K}_2).
\]
A geometric group system $(G,\mathcal{K})$ {\it splits sharply\/} if it is isomorphic 
to such a sum $(G_1,\mathcal{K}_1)\sharp(G_2,\mathcal{K}_2)$.

Suppose that for some $\kappa:S=\pi_1(Y)\to{G}$ in $\mathcal{K}$ 
there is a non-trivial element $\gamma\in\mathrm{Ker}(\kappa)$
which is represented by a separating simple closed curve $C$ on $Y$.
Then $Y/C\simeq{Y_1}\vee{Y_2}$, 
where $Y_1$ and $Y_2$ are closed surfaces,
and $\kappa$ factors through 
$S/\langle\langle\gamma\rangle\rangle\cong{S_1*S_2}$,
where $S_i=\pi_1(Y_i)$, for $i=1,2$.
Let $\kappa(i):S_i\to{G}$ be the induced homomorphisms,
and let $\mathcal{K}^\natural$ be the geometric system 
obtained by replacing $\kappa$ by the pair $\kappa(1), \kappa(2)$.
If $(G,\mathcal{K}^\natural)$ splits sharply 
(as $(G_1,\mathcal{K}_1)\sharp(G_2,\mathcal{K}_2)$, say)
we shall say that
$(G,\mathcal{K})$ {\it splits as a boundary sum}.
We may refer to $(G_1,\mathcal{K}_1)$ and 
$(G_2,\mathcal{K}_2)$ as components of $(G,\mathcal{K})$,
for either notion of splitting.

The group system arising from the third version of adding a 1-handle may be described as follows.
Let $(G,\mathcal{K})$ be a geometric group system, 
and let $\kappa_1:S_1=\pi_1(Y_1)\to{G}$ and $\kappa_2:S_2=\pi_1(Y_2)\to{G}$ 
be distinct elements of $\mathcal{K}$.
Let $Y=Y_1\#Y_2$ and let $c:Y\to{Y_1\vee{Y_2}}$ be the map 
which collapses the separating simple closed curve associated with the connected sum.
Let $\widehat{G}=G*\langle{t}\rangle$ and let $\widehat\kappa$ be the composite 
of $\pi_1(c):S=\pi_1(Y)\to{S_1}*{S_2}$ with $\kappa_1*c_t\kappa_2$, 
where $c_t$ is conjugation by $t$.
Replace the pair $\kappa_1$ and $\kappa_2$ in $\mathcal{K}$ 
by $\widehat\kappa$, to obtain $\widehat{\mathcal{K}}$.
Then we may say that $(\widehat{G},\widehat{\mathcal{K}})$ is the group system obtained by 
{\it adding a $1$-handle\/} to $(G,\mathcal{K})$.
(In the geometric situation,  adding a 1-handle includes the notion of 
boundary connected sum, but here it seems simpler to have separate definitions.)
However, we shall not use this observation.

Let $(G,\{\kappa_j|j\in{J}\})$ be a geometric group system such that 
$G\cong(*_{i\in{I}}G_i)*W$, 
where the factors $G_i$ are torsion free and have one end,   and $W$ is virtually free.
Then $\Gamma(G,\{\kappa_j\})$ is the bipartite graph with 
vertex set $I\sqcup{J}$ and an edge from $j\in{J}$ to $i\in{I}$ 
for each indecomposable factor of $\mathrm{Im}(\kappa_j)$ 
which is a $PD_2$-group and is conjugate to a subgroup of $G_i$.

\subsection*{Well built $PD_3$-pairs}

A $PD_3$-pair $(P,\partial{P})$ is {\it well-built\/} if it can be obtained 
by adding 1-handles to a connected sum of $PD_3$-complexes, 
pairs corresponding to $PD_3$-pairs of groups and copies of $(D^3,S^2)$.
Compact 3-manifolds are well-built,
but it remains an open question whether every $PD_3$-pair is well-built.


The key question in deciding whether every $PD_3$-pair is well-built
seems to be whether  \cite[Lemma 13.5]{PDD3} can be adapted to show that
if $(G,\mathcal{K})$ is a geometric group system,  
$w:G\to\Z^\times$ a homomorphism and $\mu\in{H_3(G,\mathcal{K};\Z^w)}$ 
a class which satisfies the boundary compatibility and Turaev conditions,
and if $G$ has one end then the graph $\Gamma(G,\mathcal{K})$ is a tree.
(Poincar\'e duality is used in \cite[Lemma 13.5]{PDD3}.)

If these conditions hold then the following are equivalent
\begin{enumerate}
\item{}$\Gamma(G,\mathcal{K})$ is a tree;
\item{}$\mathcal{K}$ is injective;
\item$(G,\mathcal{K})$ is the peripheral system of a $PD_3$-pair.
\end{enumerate}

In Theorem \ref{asph wellb} we shall not need to extend the validity of \cite[Lemma 13.5]{PDD3}.

\section{aspherical $PD_3$-pairs are well-built}

If $(P,\partial{P})$ is aspherical then either $(P,\partial{P})\simeq(D^3,S^2)$
or it has aspherical boundary.
(See \cite[Lemmas 3.1 and 3.11]{PDD3}.)
In each case the homotopy type of the pair is determined by the peripheral system 
$(\pi,\{\kappa_j\})$ alone.

We may assume that $\partial{P}$ is non-empty.
We note first that if $\pi_1(P)\cong\Z$ then $(P,\partial{P})\simeq(D^2\times{S^1},T)$ 
(if orientable) or $(D^2\tilde\times{S^1},Kb)$ (if non-orientable).
In each case these may be obtained from $(D^3,S^2)$ by adding a 1-handle.

\begin{lemma}
\label{asph pair}
Let $(P,\partial{P})$ be a $PD_3$-pair with peripheral system $(\pi,\{\kappa_j\})$.
If $\partial{P}$ is non-empty then the following are equivalent.
\begin{enumerate}
\item{}$P$ is aspherical;
\item{}$\cd\pi\leq2$ and $\chi(P)=\chi(\pi)$;
\item{}$\pi\cong(*^r_{i=1}G_i)*F(s)$, where $G_i$ has one end and $\cd{G_i}=2$
for $i\leq{r}$,  and $\chi(P)=1-r-s+\Sigma_{i=1}^r\chi(G_i)$.
\end{enumerate}
If these conditions hold and $\pi\not=1$ then the components of $\partial{P}$ are aspherical.
If, moreover, $s=0$ then $\Gamma(\pi,\{\kappa_j\})$ is a tree.
\end{lemma}

\begin{proof}
The implications $(1)\Rightarrow(2)\Leftrightarrow(3)$ are clear.
If (2) holds and $P$ is a finite 2-complex then $P$ is aspherical, by \cite[Theorem 2.8]{[Hi]}.
If $P$ is not homotopy equivalent to a finite 2-complex then we must modify the argument.
Let $\Gamma=\Z[\pi]$.
The equivariant chain complex for $\widetilde{P}$ is chain homotopy equivalent to
a finite complex $C_*$ of finitely generated left $\Gamma$-modules,
with $C_i=0$ for $i>3$.
The submodule $B_1=\mathrm{Im}(\partial_2)\leq{C_1}$ is projective
and finitely generated,  by Schanuel's Lemma, since $\cd\pi\leq2$.
Hence $C_2\cong{B_1}\oplus{Z_2}$ and the submodule of 2-cycles is also projective
and finitely generated.
Since $H_3(C_*)=H_3(P;\Z[\pi])=0$, by Poincar\'e duality, 
$\partial_3:C_3\to{Z_2}$ is a split monomorphism.
Hence $\Pi=\pi_2(P)=H_2(C_*)$ is also projective and finitely generated.
On tensoring with the augmentation module $\Z$ we see that $\Z\otimes_\Gamma\Pi)$
is a free abelian group of rank $\chi(P)-\chi(\pi)$, and so is 0.
But then $\Pi=0$, since the Weak Bass conjecture holds for groups of 
cohomological dimension $\leq2$ \cite{Eckm}.
Hence $P$ is aspherical.

If $P$ is aspherical and $\pi\not=1$ then all components of $\partial{P}$ are aspherical,  \cite[Lemma 3.1]{PDD3}.
If, moreover,  $s=0$ then $\chi(\Gamma(\pi,\{\kappa_j\})=1$, 
by \cite[Lemma 13.5]{PDD3}.
Suppose that $\Gamma(\pi,\{\kappa_j\})$ is not connected.
Since $P$ is aspherical we then see that $P\simeq{U}\vee{U'}$ and $\partial{P}=V\sqcup{V'}$,
with $V$ and $V'$ non-empty.
Hence $(P,\partial{P})$ has a proper decomposition as $(U\vee{U'},V\sqcup{V'})$.
But such a space cannot satisfy Poincar\'e duality,
unless either $U$ is contractible and $V$ is empty or $U'$ is contractible and $V'$ is empty.
Thus $\Gamma(\pi,\{\kappa_j\})$ is connected, and so is a tree.
\end{proof}

\begin{theorem}
\label{asph wellb}
Every aspherical $PD_3$-pair is well-built.
\end{theorem}

\begin{proof}
Let $(P,\partial{P})$ be an aspherical $PD_3$-pair.
If the peripheral system $(\pi,\{\kappa_j\})$ is injective then the result is clear,
so we may assume that $(\pi,\{\kappa_j\})$ is not injective.
We may also assume that $\pi$ is not a free group 
(since this case is settled in \cite[Theorem 3.12]{PDD3})
and $\partial{P}$ is non-empty.
(If $\pi_1(P)\cong\Z$ and $\partial{P}$ is non-empty then $(P,\partial{P})\simeq(D^2\times{S^1},T)$ 
(if orientable) or $(D^2\tilde\times{S^1},Kb)$ (if non-orientable).
In each case these may be obtained from $(D^3,S^2)$ by adding a 1-handle.)
We shall prove the theorem by successively reducing first the peripheral system 
and then the pair to simpler pieces.

The kernels of the homomorphisms $\kappa_j$ are represented by a finite family of disjoint, 
2-sided essential simple closed curves on $\partial{P}$, by the Algebraic Loop Theorem.
Suppose that $\gamma\subset{S}$ is one such curve on a boundary component $S=S_j$.
If $\gamma$ is non-separating,  then there is another simple closed curve $\xi$
which meets $\gamma$ transversely in one point.
Let $S'$ be the surface obtained from $S$ by surgery on $\gamma$.
Then $S\cong{S'\#W}$, 
where $W=T$ if $\xi$ is orientation preserving and $W=Kb$, otherwise.
We may apply \cite[Lemma 13.8]{PDD3} to see that $(\pi,\{\kappa_j\})$  splits as a boundary sum of 
$(\sigma,\{\lambda_\ell\})$ and $(\Z,E)$,
where $\pi\cong\sigma*\Z$ and $E=\Z^2$ or $\pi_1(Kb)$,
and $\{\lambda_\ell\}$ is obtained from $\{\kappa_j\}$ by replacing $\kappa_j$ by the homomorphisms 
$\kappa_j':\pi_1(S')\to\pi$ and $\kappa_W:\pi_1(W)\to\pi$ induced by $\kappa_j$.
(Note that $\mathrm{Im}(\kappa_W)\cong\Z$.)
The  aspherical $CW$-pair realizing $(\sigma,\{\lambda_\ell\})$ is then a $PD_3$-pair by
\cite[Theorem 8]{DH3}.
We may repeat this process for other non-separating curves in the family,
and reduce to a situation where $\pi\cong\omega*F(s)$ and
$(\pi,\{\kappa_j\})$ splits as an iterated boundary sum of
$(\omega,\{\lambda_\ell\})$ with $s$ terms $(\Z,E_1),\dots(\Z,E_s)$,
and there are no non-separating curves representing elements in any $\mathrm{Ker}(\lambda_\ell)$.

Suppose that $\lambda_\ell:\pi_1(S_\ell)\to\sigma$,
for some aspherical surface $V_\ell$.
Let $V=\sqcup_\ell{S_\ell}$, and let $U$ be the mapping cylinder of the map from $V$ 
to $K(\omega,1)$ which induces the homomorphisms $\lambda_\ell$.
Then $(P,\partial{P})$ is the result of adding 1-handles to $(U,V)$,
and so $(U,V)$ is an aspherical $PD_3$-pair, by \cite[Theorem 8]{DH3}.
It will clearly suffice to show that $(U,V)$ is well-built.

Thus we may reduce to the case with no simple closed curves $\gamma\subset\partial{P}$ 
which are null-homotopic in $P$,  but non-separating in $\partial{P}$.
The subgroups $\mathrm{Im}(\kappa_j)$ are then free products of $PD_2$-groups.
Therefore \cite[Lemma 13.5]{PDD3} applies, and so $\pi\cong{G}*F(n)$,
where $n=\beta_1(\Gamma(\pi,\{\kappa_j\}))$.
The graph $\Gamma(\pi,\{\kappa_j\})$ is connected,  by Lemma \ref{asph pair}.

Suppose that $\gamma$ is an essential curve on $S_j$ which has trivial image in $\pi$.
Then $S_j=S_{j1}\#S_{j2}$ is a connected sum along $\gamma$. 
Suppose that deleting the vertex $j$ and the contiguous edges 
does not separate the graph $\Gamma(\pi,\{\kappa_j\})$.
Then we may construct a new group system $(\breve\pi,\{\breve\kappa_i\})$
with $\breve\pi=G*F(n-1)$ and $\{\breve\kappa_i\}$ obtained by setting 
$\breve\kappa_i=\kappa_i$ for $i\not=j$ and replacing $\kappa_j$ by the  homomorphisms  
$\breve\kappa_{j1}:\pi_1(S_{j1})\to\pi$ and $\breve\kappa_{j2}:\pi_1(S_{j2})\to\pi$ 
induced by $\kappa_j$.
This system may be realized by a $CW$-pair $(\breve{U},\breve{Y})$
with $\breve{U}=K(\breve\pi,1)$ and $\breve{Y}$ obtained from $Y$ by replacing $S_j$
by $S_{j1}\sqcup{S_{j2}}$. 
Adding a 1-handle to $(\breve{U},\breve{Y})$ which connects $S_{j1}$ and $S_{j2}$ 
recovers our original pair (by asphericity and the Classification Theorem)
and so $(\breve{U},\breve{Y})$  is a $PD_3$-pair \cite[Theorem 8]{DH3}.

After finitely many steps we find that we may assume that $\Gamma(\pi,\{\kappa_j\})$ 
is a tree, and $n=0$. 
Similar arguments show that if some $\kappa_j$ has non-trivial kernel then 
$(\pi,\{\kappa_j\})$ is the boundary sum of two simpler group systems,
and $(P,\partial{P})$ is the corresponding boundary connected sum.
After finitely many such reductions we see that $(P,\partial{P})$ is an iterated boundary 
connected sum of pairs $(P_i,\partial{P}_i)$ with peripheral system $(G_i,\mathcal{K}_i)$.
Since $G_i$ has one end the homomorphisms in $\mathcal{K}_i$ are injective,
and so each of these group systems is a $PD_3$-pair of groups.
\end{proof}

\begin{cor}
If $P$ is aspherical and $G$ is an indecomposable factor of  $\pi=\pi_1(P)$ 
then either $G$ is the ambient group of a $PD_3$-pair of groups $(G,\Omega)$ or
$G\cong\Z$.
\qed
\end{cor}

\begin{cor}
If $P$ is aspherical then  $\pi=\pi_1(P)$ has a maximal free factor of
rank $n=\beta_1(\Gamma(\pi;\{\kappa_j\}))$.
\qed
\end{cor}

This is not true if $P$ is the internal connected sum of a $PD_3$-pair 
with a $PD_3$-complex with free fundamental group.

\begin{cor}
The First and Second Decomposition Theorems of Bleile \cite{Bleile} 
hold for all aspherical $PD_3$-pairs.
\qed
\end{cor}

Theorem \ref{asph wellb} also provides an analogue of the 3-manifold Loop Theorem,
in the following sense.
If $(P,\partial{P})$ is an aspherical  $PD_3$-pair with $\partial{P}$ a union of closed surfaces
and there is a simple closed curve $\gamma$ on $\partial{P}$ which is null-homotopic in $P$ 
then the pair is the result of adding a 1-handle to another $PD_3$-pair.
(We are not claiming here that the other pair is the result of an elementary surgery on $\gamma$ !)

\section{when does the group determine the pair?}

If $G$ is a finitely generated group then there are only finitely many homeomorphism types 
of bounded, compact, irreducible orientable 3-manifolds $M$ such that $\pi_1(M)\cong{G}$, 
and only finitely many group systems which are peripheral systems of 3-manifold pairs \cite{Sw80}.
We shall show that the latter finiteness result holds also for the peripheral systems of 
aspherical $PD_3$-pairs $(P,\partial{P})$ such that $\partial{P}$ is non-empty,
provided that $\chi(\pi)=1-n$, 
where $\pi_(P)$ has $n$ irreducible factors.

If $(X,Y)$ is a $PD_3$-pair then $\chi(Y)=2\chi(X)$,  $\beta_1(Y)\leq2\beta_1(X)$ and
\[
\chi(X)\leq\beta_0(Y;\mathbb{F}_2)\leq1+\beta_2(X;\mathbb{F}_2)=\chi(X)+\beta_1(X;\mathbb{F}_2),
\]
by Poincar\'e duality and the exact sequence of homology.
If the components of $Y$ are aspherical then $\chi(X)\leq0$,  
and so $1\leq\beta_0(Y;\mathbb{F}_2)\leq\beta_1(\pi;\mathbb{F}_2)$.
Hence (given $\pi$) there are only finitely many possibilities for $Y$.
Thus we may assume $Y$ is given.

A $PD_3$-pair of groups $(G,\Omega)$ is {\it of Seifert type\/} if $\sqrt{G}\not=1$.
It is {\it properly\/} of Seifert type if $G$ is not polycyclic.
If $\Omega=\emptyset$ then $G\cong\pi_1(M)$ for some aspherical closed 3-manifold $M$  \cite{Bowd}.
If $(G,\Omega)$ is properly of Seifert type and $\Omega$ is non-empty then 
$\cd{G}=2$, $\zeta{G}\cong\Z$,  $G'$ is free and $G'\cap\zeta{G}=1$ \cite{[Bi]}.
Moreover,  we then have $G=\pi\mathcal{G}$ for some graph of groups $(\Gamma, \mathcal{G})$
with all edge and vertex groups $\Z$ \cite{St76}.

\begin{lemma}
If $r\geq1$ and $s=\lfloor\frac{r}2\rfloor$ then there are $s+1$ orientable and
$r$ non-orientable $PD_2$-pairs $(X,Y)$ with $\pi_1(X)\cong{F(r)}$.
\end{lemma}

\begin{proof}
Let $b=\beta_0(Y)$.
Then $b>0$,  since $\pi_1(X)$ is free, 
and $b\leq{r+1}$, by a simple homological argument.
Adding $D^2$ along one component of $Y$ reduces each of $b$ and $r$ by 1.
If $b=1$ then capping off $Y$ gives a $PD_2$-complex $X/Y$ with
with $\beta_1(X/Y;\mathbb{F}_2)=r$.
Hence $b\leq{r+1}$.
If $(X,Y)$ is orientable then $\beta_1(X/Y;\mathbb{F}_2)$ is even and so $b\equiv{r+1}~mod~(2)$.
All of these possibilities may be realized by punctured surfaces, 
and there are no others \cite{EM80}.
 \end{proof}

Taking products with $S^1$ gives examples of $PD_3$-pairs which are not determined
by the ambient group $\pi$.
In particular,
if $T_o$ and $D_{oo}$ are the punctured torus and the twice-punctured disc, 
then $\pi_1(T_o\times{S^1})\cong\pi_1(D_{oo}\times{S^1})\cong{F(2)}\times\mathbb{Z}$.

\begin{lemma}
\label{virtualisom}
Let $G$ be a  group and $\Omega$ a left $G$-set with finitely many orbits.
Suppose that $H<G$ is a normal subgroup of finite index.
Then there are only finitely many $G$-isomorphism classes of left $G$-sets $\Omega'$
such that $\Omega'|_H$ is $H$-isomorphic to $\Omega|_H$.
\end{lemma}

\begin{proof}
Suppose that  $[G:H]=n$, with coset representatives $\{g_1,\dots,g_n\}$.
If  $\Omega=\sqcup_{j=1}^kGx_j$ and  $\Omega'=\sqcup_{j=1}^{k'}Gx'_j$
then $\Omega|_H=\sqcup{H}g_ix_j$ and $\Omega'|_H=\sqcup{H}g_ix'_j$. 
Hence $nk=nk'$ and so $k=k'$. 
An $H$-isomorphism $f:\Omega|_H\cong\Omega'|_H$ may permute the indices, 
but up to a finite ambiguity we may assume that $f(g_ix_j)=g_ih_{i,j}x'_j$, for some $h_{i,j}\in{H}$.
The bijection $hg_ih_{i,j}x'_j\mapsto{h}g_ix'_j$ is an $H$-isomorphism, 
and so we may replace $f$ by an $H$-isomorphism $\tilde{f}$
such that $\tilde{f}(hg_ix_j)=hg_ix'_j$ for all $h\in{H}$.
Then  $\tilde{f}$ is a $G$-isomorphism, since $G=\sqcup_{i=1}^nHg_i$.
\end{proof}

Some features of \cite[Theorem 8.8]{PDD3} are used in the next lemma.

\begin{lemma}
\label{seifert}
If $(G,\Omega)$ is of Seifert type then there are only finitely many
distinct $PD_3$-pairs of groups $(\widehat{G},\widehat\Omega)$
with $\widehat{G}\cong{G}$,  modulo automorphisms of $G$.
\end{lemma}

\begin{proof}
We may assume that $(G,\Omega)$ is properly of Seifert type,
and so $A=\sqrt{G}\cong\Z$.
We may also assume that $\Omega$ is non-empty, and so $\cd{G}=2$.
If $S$ is a boundary component then $A\cap{S}\not=1$, for otherwise $c.d.AS=3$.
Since $A\cap{S}\cong\mathbb{Z}$ is normal in $S$,
it follows that $S\cong\mathbb{Z}^2$ or $Kb$.
We also have $[G:S]=\infty$, since $G$ is not polycyclic.

Let $C=C_G(A)$. 
Then $[G:C]\leq2$, and $C'$ is a nonabelian free group,
since $A$ is central in $C$ and $G$ is properly of Seifert type \cite[Theorem 8.8]{[Bi]}.
Hence $A\cap{C}=1$, and so there is an epimorphism $\psi:C\to\Z$ 
which is injective on $A$.
Let $N=\mathrm{Ker}(\psi)$.
Then $NA\cong{N}\times\Z$ has finite index in $G$ and so
$N$ is finitely generated.
Hence $N$ is also free \cite[Corollary 8.6]{[Bi]}.

If $(G,\Omega)$ is a $PD_3$-pair of groups
then $N$ is the ambient group of a $PD_2$-pair of groups 
$(N,\Psi)$, and $\alpha$ preserves the boundary.
Since there are only finitely $PD_2$-pairs $(N,\Psi)$
with $\pi_1(N)\cong{F(r)}$, 
it follows that there are only finitely many $PD_3$-pairs 
with ambient group $NA\cong{N}\times\Z$.
Hence there are only finitely many $PD_3$-pairs of groups with ambient group $G$,
by Lemma \ref{virtualisom}.
\end{proof}

A $PD_3$-pair of groups $(G,\Omega)$ is {\it atoroidal\/} if $G$ has one end,
is not virtually $\Z^2$ and every subgroup $H<G$ such that $H\cong\Z^2$ or 
$Kb$ fixes a point of $\Omega$.


\begin{lemma}
\label{atoroidal}
Let $(G,\Omega)$ be an atoroidal $PD_3$-pair of groups such that $\chi(G)=0$.
If $(G,\widehat\Omega)$ is another $PD_3$-pair of groups then 
$(G,\widehat\Omega)\cong(G,\Omega)$.
\end{lemma}

\begin{proof}
Since $\chi(G)=0$ the boundary components of  $(G,\Omega)$ and 
$(G,\widehat\Omega)$ are copies of $\Z^2$ and $\pi_1(Kb)$.
They are their own commensurators in $G$,
and are pairwise non-conjugate \cite{KR88a}, \cite[Lemma 10.3]{PDD3}.
Every subgroup isomorphic to $\Z^2$ or $\pi_1(Kb)$ fixes a point of $\Omega$,
since $(G,\Omega)$ is atoroidal, and so the boundary components of $(G,\widehat\Omega)$
are conjugate in $G$ to boundary components of $(G,\Omega)$.

Let $b$ and $\widehat{b}$ be the number of boundary components of $(G,\Omega)$
and $(G,\widehat\Omega)$, respectively.
The images of the fundamental classes of the boundary components of $(G,\Omega)$
in $H_2(G;\mathbb{F}_2)$ have sum 0 and generate a subspace of dimension $b-1$,  
by the boundary compatibility condition.
Since similar conditions hold for the boundary components of  $(G,\widehat\Omega)$, 
the boundaries must correspond bijectively (and $\widehat{b}=b$).
Klein bottles must correspond to Klein bottles, 
since the boundary components are their own commensurators in $G$.
Hence $\widehat\Omega\cong\Omega$ as $G$-sets, and so 
$(G,\widehat\Omega)\cong(G,\Omega)$.
\end{proof}

Thus the notion of atoroidal depends only on the ambient group $G$ if $\chi(G)=0$.

We shall use the above results with
Kropholler' JSJ Decomposition Theorem for $PD_n$-pairs of groups.
This was the precursor for other  analogues of the JSJ Decomposition Theorems 
for groups and pairs of groups,
and was originally formulated for pairs $(G,\Omega)$ in which $G$ has  {\it max-c\/};
increasing chains of centralizers are finite \cite{Krop}.
Although this property does not hold for all  $PD_n$-pairs of groups,
it has since been established for $PD_3$-pairs.
(See \cite[Theorem 9.5]{PDD3} and the subsequent remarks.)

\begin{lemma}
\label{indec chi=0}
Let $G$ be a finitely presentable, indecomposable group such that $\cd{G}=2$ and $\chi(G)=0$.
Then there are only finitely many homotopy types of 
$PD_3$-pairs $(P,\partial{P})$ with $\pi_1(P)\cong{G}$ and $\chi(P)=0$.
\end{lemma}

\begin{proof}
Let $(P,\partial{P})$ be such a $PD_3$-pair.
Then $\partial{P}$ is non-empty, since $\cd{G}=2$.
Hence  $P$ and the components of $\partial{P}$ are aspherical,  by Lemma \ref{asph pair}.

The boundary components are tori or Klein bottles,
since $\chi(G)=0$, and are $\pi_1$-injective, 
since $P$ is aspherical and $G$ is indecomposable, 
and so has one end  \cite[Lemma 3.1]{PDD3}.
Hence $(P,\partial{P})$ is the pair corresponding to a $PD_3$-pair of groups $(G,\Omega)$.
Kropholler's JSJ decomposition of $G$ is preserved under automorphisms of $G$, 
and each of the pieces $(G_v,\Omega_v)$ of the decomposition is determined 
up to finite ambiguity by the above lemmas.
It follows easily that $(P,\partial{P})$ is so determined.
\end{proof}

The corresponding result for irreducible and $\partial$-irreducible 3-manifolds
given in \cite{Sw80} has no restriction on the Euler characteristics arising.
In general, a boundary component $G_\omega$ is a maximal $PD_2$-subgroup.
At present we lack an argument that might show that there are only finitely many conjugacy classes of maximal $PD_2$-subgroups $H$ of given Euler characteristic $\chi(H)<0$.

\begin{theorem}
\label{finitely many}
Let $G_i$ be an indecomposable $FP_2$ group such that
$\cd{G_i}=2$ and $\chi(G_i)=0$,  for $i\leq{r}$,
and let $\pi=(*^r_{i=1}G_i)*F(s)$.
Then there are only finitely many homotopy types of
$PD_3$-pairs $(P,\partial{P})$ with $\pi_1(P)\cong\pi$ and
$\chi(P)=\chi(\pi)=1-r-s$.
\end{theorem}

\begin{proof}
Suppose that $(P,\partial{P})$ is such a $PD_3$-pair.
If $r=0$ then $\pi=F(s)$ and the result follows from Theorem 3.12 of \cite{PDD3}.
(If $s>0$ there are two homotopy types, one orientable and one not.)
Thus we may assume that $r>0$.
The ambient space $P$ and the components of $\partial{P}$ are aspherical
by Lemma \ref{asph pair},
and $(P,\partial{P})$ may be obtained by attaching 1-handles to a $PD_3$-pair 
$(Q,\partial{Q})\simeq\natural_{i=1}^r(Q_i,\partial{Q}_i)$,
with $Q_i$ aspherical and  $\pi_1(Q_i)\cong{G_i}$,
since the pair is well-built,
by Theorem \ref{asph wellb}.

There are only finitely many possibilities for each 
of these $PD_3$-pairs, by Lemma \ref{indec chi=0}.
Since there are only finitely many ways of forming the boundary connected sum of two
given $PD_3$-pairs, this gives the theorem.
\end{proof}


\newpage
\begin{thebibliography}{99}

\bibitem[Bi]{[Bi]} Bieri, R. 
\textit{ Homological Dimensions of Discrete Groups}, 

Queen Mary College Mathematical Notes, London (1976).

\bibitem[Hi]{[Hi]} Hillman, J. A. 
\textit{Four-Manifolds, Geometries and Knots},

Geometry and Topology Monographs 5, 
Geometry and Topology Publications (2002, 2007). 

(Latest revision: arXiv 0212142.v2 math.GT,  Nov.  2022.)

\bibitem[PDD3]{PDD3} Hillman, J. A.  \textit{Poincar\'e Duality in Dimension 3},

2nd edition, The Open Book Series, MSP, Berkeley (2023).

\bibitem{Bleile} Bleile, B. Poincar\'e Duality Pairs of Dimension Three,

Forum Math. 22 (2010), 277--301.

\bibitem{Bowd} Bowditch, B. H. Planar groups and the Seifert conjecture,

J. Reine Angew. Math. 576 (2004), 11--62.

\bibitem{DH3} Davis, J. and Hillman, J. A.  Aspherical $PD_4$-pairs,

in preparation.

\bibitem{Eckm} Eckmann, B.  Cyclic homology of groups and the Bass Conjecture,

Comment. Math. Helv. 61 (1986), 193--202.

\bibitem{EM80} Eckmann, B. and M\"uller, H. Poincar\'e duality groups of dimension two,

Comment. Math. Helv. 55 (1980), 510--520.

\bibitem{Krop} Kropholler, P. An analogue of the torus decomposition theorem 
for certain Poincar\'e duality groups,
Proc. London Math. Soc. 60 (1990), 503--529.

\bibitem{KR88a} Kropholler, P. H. and Roller, M. A. Splittings of Poincar\'e 
duality groups II, 

J. London Math. Soc. 38 (1988), 410--420.

\bibitem{RSS} Reeves, L., Scott, P. and Swarup, G.  Assoc. Math. Res. Monographs vol.3,

Association for Mathematical Research, Davis CA (2023).

\bibitem{St76} Strebel, R. A homological finiteness theorem,

Math. Z. 151  (1976), 263--275.

\bibitem{Sw80} Swarup, G.  A.  Two finiteness properties in 3-manifolds,

Bull. London Math. Soc.  12 (1980), 296--302.

\bibitem{Wa67} Wall, C. T. C.  Poincar\'e complexes: I,

Ann. Math. 86 (1967), 213--245.

\end{thebibliography}
\end{document}